\documentclass[a4paper,10pt]{amsart}
\usepackage{amsmath,amssymb,amsthm,amscd}
\usepackage[all]{xy}

\newcommand{\st}{such that }

\newcommand{\RMod}{\hbox{R{\rm -Mod}}}
\newcommand{\ModR}{\hbox{{\rm Mod-}R}}

\newcommand{\add}{\mathrm{add}}
\newcommand{\Add}{\mathrm{Add}}

\newcommand{\p}{\mathfrak{p}}
\newcommand{\q}{\mathfrak{q}}
\newcommand{\tube}{\mathfrak{t}}

\DeclareMathOperator{\Hom}{Hom}

\DeclareMathOperator{\End}{End}
\DeclareMathOperator{\Ext}{Ext}
\DeclareMathOperator{\Tor}{Tor}
\DeclareMathOperator{\Ker}{Ker}
\DeclareMathOperator{\Img}{Im}

\usepackage[usenames]{color}

\theoremstyle{plain}
\newtheorem{thm}{Theorem}[section]
\newtheorem{prop}[thm]{Proposition}
\newtheorem{lem}[thm]{Lemma}
\newtheorem{cor}[thm]{Corollary}
\theoremstyle{definition}
\newtheorem{defn}[thm]{Definition}
\newtheorem{exm}[thm]{Example}

\theoremstyle{remark}
\newtheorem{rem}{Remark}

\begin{document}
\title[Approximations and ML conditions -- applications]{Approximations and Mittag-Leffler conditions \quad the applications}

\author{\textsc{Lidia Angeleri H\" ugel}}
\address{Dipartimento di Informatica -
Settore di Matematica\\ 
Universit\`a degli Studi di Verona\\
Strada le Grazie 15 - Ca' Vignal\\
37134 Verona, Italy}
\email{lidia.angeleri@univr.it}

\author{\textsc{Jan \v Saroch}}
\address{Charles University, Faculty of Mathematics and Physics, Department of Algebra \\ 
Sokolovsk\'{a} 83, 186 75 Praha~8, Czech Republic}
\email{saroch@karlin.mff.cuni.cz}

\author{\textsc{Jan Trlifaj}}
\address{Charles University, Faculty of Mathematics and Physics, Department of Algebra \\ 
Sokolovsk\'{a} 83, 186 75 Praha~8, Czech Republic}
\email{trlifaj@karlin.mff.cuni.cz}
 
\keywords{Mittag-Leffler conditions, approximations of modules, tilting module, cotorsion pair, pure-injective module, deconstructible class, locally $T$-free modules, Bass module, mono-orbit.}

\thanks{The research of Angeleri H\"ugel has been supported by  DGI MICIIN
MTM2011-28992-C02-01,  by  Generalitat de Catalunya through Project 2009 SGR
1389, and by Fondazione Cariparo, Progetto di Eccellenza ASATA.
The research of \v Saroch and Trlifaj supported by grant GA\v CR 14-15479S}

\date{\today}

\begin{abstract} 
A classic result by Bass says that the class of all projective modules is covering, if and only if it is closed under direct limits. Enochs extended the if-part by showing that every class of modules $\mathcal C$, which is precovering and closed under direct limits, is covering, and asked whether the converse is true. We employ the tools developed in \cite{S} and give a positive answer when $\mathcal C = \mathcal A$, or $\mathcal C$ is the class of all locally $\mathcal A ^{\leq \omega}$-free modules, where $\mathcal A$ is any class of modules fitting in a cotorsion pair $(\mathcal A, \mathcal B)$ such that $\mathcal B$ is closed under direct limits. This setting includes all cotorsion pairs and classes of
locally free modules arising in (infinite-dimensional) tilting theory.
 We also consider two particular applications: to pure-semisimple rings, and artin algebras of infinite representation type. 
%
%
\end{abstract}

\maketitle
\vspace{4ex}

\section{Introduction} 
The additive closure $\Add (M)$ of a module $M$ over a ring $R$ is always a precovering class. Bass' Theorem P deals with the existence of minimal right $\Add (M)$-approximations when $M=R$. More generally, when $M$ is a direct sum of finitely presented modules, the existence of minimal right $\Add (M)$-approximations is equivalent to $\Add (M)$ being closed under direct limits, and it  can be rephrased by a~descending chain condition over the endomorphism ring of $M$, see \cite[Theorem 4.4]{A}. In this paper, we will prove the same result for modules occurring as additive generators of the kernel of certain cotorsion pairs, including the ones studied in (infinite dimensional) tilting theory.

More precisely, we will consider cotorsion pairs $(\mathcal A, \mathcal B)$ with the right-hand class being closed under direct limits. It  was proved in \cite{S} that such cotorsion pairs are always complete and of countable type, and that the class $\mathcal B$ is even definable, that is, it is closed  under pure submodules, in addition to direct products and direct limits. We are going to show that the kernel $\mathcal A\cap\mathcal B$ is of the form $\Add (M)$ for some module $M$, and that $M$ has the properties discussed above if and only if the class $\mathcal A$ is closed under direct limits, or equivalently,  $(\mathcal A, \mathcal B)$ is a perfect cotorsion pair, i.e.~it yields  $\mathcal A$-covers and  $\mathcal B$-envelopes (Corollary~\ref{c:P}). 

The key to prove these results is  a reduction  to the countable case.   Indeed, the tools developed in  \cite{S} allow us to test for approximation properties on a particular class of countably presented modules, the Bass modules. They further provide a~useful connection with the class $\mathcal L$ of all locally free modules (with respect to $\mathcal A^{\le \omega}$). This class  is located between $\mathcal A$ and its direct limit closure, and it can be described in terms of a Mittag-Leffler condition. It turns out that  the cotorsion pair   $(\mathcal A, \mathcal B)$ is perfect if and only if $\mathcal L$ is  deconstructible, or equivalently,  every module has an $\mathcal L$-precover.

In particular, our  results apply to the cotorsion pairs  $(\mathcal A, \mathcal B)$ where $\mathcal B=T^{\perp_\infty}$ for some tilting module $T$.  In this case,  an additive generator of the kernel is provided by $T$ itself. It follows that $(\mathcal A, \mathcal B)$ is perfect  if and only if every pure submodule of a direct sum of copies of $T$ is a direct summand.  In some cases, e.g.~when the ring is noetherian and $T$ has projective dimension at most one,  this amounts to $T$ being $\sum$-pure-injective and even product-complete.

The paper is organized as follows. After some preliminaries in Section \ref{sec:prelim}, we start investigating cotorsion pairs $(\mathcal A, \mathcal B)$ with both classes being closed under direct limits in Section~\ref{sec:closed}. Section~\ref{sec:lfree} is devoted to locally free modules. Section~\ref{sec:covers} contains our main results discussed above.  In Section~\ref{sec:pss}, we derive some consequences  related to the Pure-Semisimplicity Conjecture. Section~\ref{sec:herartinalg} exhibits an explicit example of a Bass module $N$ over a hereditary artin algebra of infinite representation type such that $N$ does not possess a locally Baer precover.

\section{Preliminaries}
\label{sec:prelim}

We will freely use the terminology introduced in \cite{S}. Here we collect some further notions needed in the sequel.

\begin{defn}\label{d:tilt} A module $T$ is \emph{tilting}, provided $T$ has finite projective dimension, $\Ext ^i_R(T,T^{(I)}) = 0$ for each $i \geq 1$ and each set $I$, and there exist a $k < \omega$ and an exact sequence $0 \to R \to T_0 \to \dots \to T_k \to 0$ such that $T_i \in \Add (T)$ for each $i \leq k$. Each tilting module induces the \emph{tilting class} $\mathcal B = T^{\perp_\infty} = \bigcap_{1 \leq i}\Ker \Ext ^i_R(T,-)$. Moreover, $T$ is called \emph{$n$-tilting} in case $T$ has projective dimension $\leq n$.
\end{defn}

With each tilting module $T$, a cotorsion pair and a class of locally free modules are associated, as follows:
If $T$ is a tilting module with the induced  tilting class $\mathcal B$, then there is a hereditary cotorsion pair $\mathfrak C =  (\mathcal A, \mathcal B)= (^\perp \mathcal B,\mathcal B)$, called the \emph{tilting cotorsion pair} induced by $T$. 
The kernel of $\mathfrak C$ equals $\Ker \mathfrak C = \Add (T)$, and $\mathfrak C$ is of finite type. In particular, the class $\mathcal A$ is $\aleph_1$-deconstructible, and $\mathcal B$ is closed under direct limits.     

The 
locally $\mathcal A^{\le \omega}$-free modules (in the sense of \cite[Definition 2.1]{S}) are called the \emph{locally $T$-free} modules. 

\begin{exm} (i) If $T$ is the $0$-tilting module $R$, then the locally $T$-free modules coincide with the flat Mittag-Leffler modules \cite{HT}. Indeed, one of our main goals here will be to extend the results proved in \cite{S} for flat Mittag-Leffler modules to the general setting of locally $T$-free modules for an arbitrary tilting module $T$.

(ii) If $R$ is a hereditary artin algebra of infinite representation type and $T$ is the Lukas tilting module, then the locally $T$-free modules are called the \emph{locally Baer modules}, \cite{SlT}. This example will be considered in more detail in Section~\ref{sec:herartinalg}.   
\end{exm}

 Given a class $\mathcal B$ of modules closed under direct limits and products, we know by \cite[Lemma 5.3]{S}  that $\mathcal B$ contains a pure-injective module $C$ such that each $B\in\mathcal B$ can be purely embedded into a direct product of copies of $C$, see also \cite[5.3.52]{P}. We will call any such $C\in\mathcal B$ \emph{an elementary cogenerator for $\mathcal B$}. 

The pure-injective hull of a module $M$ will be denoted by $PE(M)$. A pure-injective module is \emph{discrete} if it is isomorphic to the pure-injective hull of a direct sum of indecomposable pure-injective modules.

Moreover, a module $M$ is called \emph{$\Sigma$-pure-split} if for all $N\in\Add (M)$, any pure embedding into $N$ splits. Each $\sum$-pure-injective module is $\Sigma$-pure-split, but the converse fails in general (see Section 5).

 \medskip
 
Before we proceed, let us recall several important results from \cite{S}. First of all, if $ (\mathcal A, \mathcal B)$ is a cotorsion pair with $\mathcal B$ being closed under direct limits, then it follows from \cite[Lemma 4.2]{S} that every module in $\mathcal A$ is strict $\mathcal B$-stationary.

We are going to use the following lemma many times. 

\begin{lem}\label{l:Cext} \cite[Lemma 5.4]{S} Let $C$ be a pure-injective module which cogenerates $\ModR$, and $M$ be a strict $C$-stationary module. Then there exists an $\aleph_1$-dense system $\mathcal L$ of strict $C$-stationary submodules of $M$ \st  $$\Hom_R(M,C)\to\Hom_R(N,C)\hbox{ is surjective}\eqno{(\dagger)}$$ 
for every directed union $N$ of modules from $\mathcal L$. 
\end{lem}

In the lemma below, $X^c$ stays for the usual character module of $X\in\ModR$, i.e. $X^c = \Hom_{\mathbb Z}(X,\mathbb Q/\mathbb Z)\in\RMod$.

\begin{lem} \label{l:kerepi} \cite[Lemma 4.4]{S} Let $B$ be a module and $0\to N 
\to A\to M\to 0$ a~short exact sequence of modules such that  $M\in{}^\perp (B^{(I)})^{cc}$ for each set $I$. Then $N$ is strict $B$-stationary if so is $A$. 
\end{lem}

Finally, we need a variant of \cite[Proposition 2.7]{SS}:

\begin{prop} \label{p:pureinprod} Let $\mathcal G$ be a class of modules and $M$ a countably presented module such that $M\in {}^\perp \mathcal G$. Then the following is equivalent:
\begin{enumerate}
\item $M\in {}^\perp D$ for each module $D$ isomorphic to a pure submodule of a product of modules from $\mathcal G$;
\item $M\in {}^\perp D$ for each module $D$ isomorphic to a countable direct sum of modules from $\mathcal G$;
\item $M$ is $\mathcal G$-stationary.
\end{enumerate}
\end{prop}

\section{Closed cotorsion pairs}
\label{sec:closed}

In this section, we will characterize the tilting cotorsion pairs $\mathfrak C = (\mathcal A,\mathcal B)$ such that $\mathfrak C$ is closed, that is, $\varinjlim \mathcal A = \mathcal A$. In fact, in Theorem~\ref{t:bigchar1} we will go far beyond the tilting setting: we will not require $\mathfrak C$ to be hereditary or $\mathcal A$ to have bounded projective dimension. Further characterizations for the closure of $\mathfrak C$ will be given later in Theorem~\ref{t:bigchar2} and Corollary~\ref{c:P}.

First, we recall the following important result going back to Hill (in the form presented in \cite[Theorem 7.10]{GT}, for example):
 
\begin{lem} \label{l:hill}
Let $\lambda$ be a regular infinite cardinal. Let $\mathcal S$ be a class of $< \lambda$-presented modules and $M$ a module possessing an $\mathcal S$-filtration $(M_\alpha \mid \alpha\leq\sigma)$. Then there is a family $\mathcal F$ of submodules of $M$ such that:
\begin{enumerate}
\item $M_\alpha \in \mathcal F$ for all $\alpha\leq\sigma$.
\item $\mathcal F$ is closed under arbitrary sums and intersections.
\item For each $N,P \in \mathcal F$ \st $N \subseteq P$, the module $P/N$ is $\mathcal S$-filtered.
\item For each $N \in \mathcal F$ and a subset $X \subseteq M$ of cardinality $< \lambda$, there is $P \in \mathcal F$ \st $N \cup X \subseteq P$ and $P/N$ is $< \lambda$-presented.\qed
\end{enumerate}
\end{lem}
\medskip

We will also capitalize on a useful description of the class of pure-epimorphic images of modules from $\mathcal A$, cf.~\cite[Lemmas 8.38 and 8.39]{GT} and \cite[Proposition 5.12]{SS}.

\begin{lem}\label{l:tilde} Let $R$ be a ring and $\mathfrak C=(\mathcal A,\mathcal B)$ a cotorsion pair in $\ModR$ \st $\mathcal B$ is closed under direct limits, and let $\mathcal B'$ be the class of all pure-injective modules  from $\mathcal B$. Then ${}^\perp \mathcal B'$ coincides with the class $\tilde{\mathcal A}$ of all pure-epimorphic images of modules from $\mathcal A$.
\end{lem}

Having a cotorsion pair $(\mathcal A, \mathcal B)$ with $\mathcal B = \varinjlim\mathcal B$, the next two lemmas imply that any pure-epimorphic image of a module from $\mathcal A$ is the direct limit of a direct system of Bass modules over $\mathcal A ^{\leq\omega}$.

\begin{lem} \label{l:limpure-epi} Let $(\mathcal A, \mathcal B)$ be a cotorsion pair with $\mathcal B = \varinjlim\mathcal B$. Denote by $\tilde{\mathcal A}$ the class of all pure-epimorphic images of modules from $\mathcal A$. Then $\tilde{\mathcal A} = \varinjlim {\tilde{\mathcal A}}^{\leq\omega}$.
\end{lem}

\begin{proof} Denote by $C$ an elementary cogenerator for the class $\mathcal B$. From Lemma~\ref{l:tilde} and the properties of $C$, we get $\tilde{\mathcal A} = {}^\perp C$. Let $F\in \tilde{\mathcal A}$ be a (non-countably presented) module. Take any presentation $$\begin{CD} 0 @>>> K @>{\subseteq}>> R^{(X)} @>>> F @>>> 0 \end{CD}$$ with $X$ infinite. Then $K$ is (strict) $C$-stationary by Lemma~\ref{l:kerepi}.

Let $\mathcal L$ be an $\aleph_1$-dense system from Lemma~\ref{l:Cext}. Put $\mathcal I = \{(L,Y)\in\mathcal L\times [X]^\omega \mid L\subseteq R^{(Y)}\}$, where $[X]^\omega$ denotes the set of all countable subsets of $X$. Then $\mathcal I$, together with inclusion in both coordinates, is a directed poset. The module $F$ is the direct limit of the induced direct system $(R^{(Y)}/L \mid (L,Y)\in\mathcal I)$.

It follows from the properties of the system $\mathcal L$ that $R^{(Y)}/L\in\tilde{\mathcal A}$ whenever $(L,Y)\in \mathcal I$. Indeed, $\tilde{\mathcal A} = {}^\perp C$ and any homomorphism $h$ from $L\in\mathcal L$ to $C$ can be extended to an element of $\Hom_R(K,C)$, and then further to a homomorphism $R^{(X)}\to C$. The restriction of the latter map to $R^{(Y)}$ (where $L\subseteq R^{(Y)}$) is an extension of $h$, whence $R^{(Y)}/L\in {}^\perp C$.
\end{proof}

Notice that the direct system in the proof above is even \emph{$\aleph_1$-continuous}, i.e. it is closed under taking direct limits of its countable direct subsystems.

\begin{lem} \label{l:countlim} Let $\mathfrak C = (\mathcal A, \mathcal B)$ be a complete cotorsion pair with $\mathcal B$ closed under countable direct limits. Assume that $M$ is a countably presented pure-epimorphic image of a module from $\mathcal A$. 

Then there exists a direct system $(F_n, f_{nm} \mid m \leq n<\omega)$ of finitely presented modules \st $M = \varinjlim_{n<\omega}F_n$, and $f_{n+1,n}$ factors through a module from $\mathcal A$ for each $n<\omega$. In particular, $M$ is a countable direct limit of modules from $\mathcal A$.
 
If moreover $\mathfrak C$ is of countable type (finite type), then $M$ is a Bass module over $\mathcal A ^{\leq \omega}$ ($\mathcal A ^{< \omega}$).  
\end{lem}
\begin{proof} Let $\mathcal D = (F_n, f_{nm} \mid m < n < \omega )$ be a direct system of finitely presented modules with the direct limit $( M, f_n \mid n < \omega)$. We can expand $\mathcal D$ to a direct system of special $\mathcal A$-precovers $\pi_n : A_n \to F_n$ of the modules $F_n$ ($n < \omega$) so that the diagram
$$\begin{CD}
\dots	@>>> A_n @>{g_n}>> A_{n+1} @>>> \dots \\ 
@.		@V{\pi_n}VV	 @V{\pi_{n+1}}VV  @.\\
\dots	@>{f_{n,n-1}}>> F_n @>{f_{n+1,n}}>> F_{n+1}  @>{f_{n+2,n+1}}>> \dots
\end{CD}$$
is commutative. Then $\pi = \varinjlim_{n< \omega} \pi_n$ is a pure epimorphism: indeed, as $\Ker(\pi) \in \mathcal B$ by our assumption on $\mathcal B$, the presentation of $M$ as a pure-epimorphic image of a module from $A$ factors through $\pi$. Since the modules $F_n$ are finitely presented, by possibly dropping some of them, we can assume that $f_{n+1,n} = \pi_{n+1} \nu_n$ where $\nu_n \in \Hom _R(F_n,A_{n+1})$ for each $n < \omega$. Then $( M, f_n\pi_n \mid n < \omega )$ is the direct limit of the direct system $( A_n, g_{nm} \mid m \leq n < \omega )$ where $g_{nm} = \nu_{n-1}\pi_{n-1} \dots \nu_m\pi_m$ for all $m < n < \omega$.     
 
If $\mathfrak C$ is of countable type, then each $A_n$ is $\mathcal A ^{\leq \omega}$-filtered.  
We use Lemma \ref{l:hill}, for $\lambda = \aleph_1$, to build inductively, for each $n$, a submodule $A_n^\prime \in \mathcal A ^{\leq \omega}$ of $A_n$  which contains (at most countable) generating sets of $\Img(\nu_{n-1})$, $g_{n-1}(A^\prime_{n-1})$ as well as of a finitely generated module $G$ such that $G+\Ker(\pi_n)=A_n$. We replace each $A_n$ by  $A_n^\prime$, and $\pi_n$ and $g_n$ by their restrictions $\pi_n^\prime$ and $g_n^\prime$ to $A_n^\prime$, respectively.
So  the diagram  
$$\begin{CD}
\dots	@>{g_{n-1}^\prime}>> A_n^\prime @>{g_n^\prime}>> A_{n+1}^\prime @>{g_{n+1}^\prime}>> \dots \\ 
@.		@V{\pi_n^\prime}VV	 @V{\pi_{n+1}^\prime}VV  @.\\
\dots	@>{f_{n,n-1}}>> F_n @>{f_{n+1,n}}>> F_{n+1}  @>{f_{n+2,n+1}}>> \dots
\end{CD}$$          
is commutative. As above, $( M, f_n\pi_n^\prime \mid n < \omega )$ is the direct limit of the direct system $( A_n^\prime, g_{nn}^\prime \mid m \leq n < \omega )$ where $g_{nm}^\prime = \nu_{n-1}\pi_{n-1}^\prime \dots \nu_m\pi_m^\prime$ for all $m < n < \omega$. In particular, $M$ is a Bass module over 
$\mathcal A ^{\leq \omega}$. 

If $\mathfrak C$ is of finite type, then  each $A_n$ is a direct summand in a $\mathcal A ^{< \omega}$-filtered module with a complement in $\Ker (\mathfrak C) = \mathcal A \cap \mathcal B$ (see \cite[Corollary 6.13(b)]{GT}). Thus, we can w.l.o.g.\ assume that in the special $\mathcal A$-precover $\pi_n : A_n \to F_n$, the module $A_n$ is $\mathcal A ^{< \omega}$-filtered. Using Lemma \ref{l:hill} for $\lambda = \aleph_0$, we replace each $A_n$ by its submodule $A_n^\prime \in \mathcal A ^{< \omega}$, and $\pi_n$ and $g_n$ by their restrictions $\pi_n^\prime$ and $g_n^\prime$ to $A_n^\prime$, respectively, so that $\Img(\nu_n) \subseteq A^\prime_{n+1}$, and the diagram above is commutative. As above, we conclude that $M$ is a Bass module over $\mathcal A ^{< \omega}$. 
\end{proof} 

\begin{cor} \label{c:Cstat} Let $(\mathcal A, \mathcal B)$ be a cotorsion pair with $\mathcal B = \varinjlim\mathcal B$, and let $\mathcal C$ be any class of modules. Assume that all Bass modules over $\mathcal A^{\leq\omega}$ are $\mathcal C$-stationary. Then all pure-epimorphic images of modules from $\mathcal A$ are $\mathcal C$-stationary.
\end{cor}

\begin{proof} Combining the two lemmas above, we can express any pure-epimorphic image of a module from $\mathcal A$ as the direct limit of an $\aleph_1$-continuous direct system consisting of Bass modules over $\mathcal A^{\leq\omega}$. Using \cite[Corollary 2.6(3)]{H}, we know that a module is $\mathcal C$-stationary, if and only if it is $M^c$-Mittag-Leffler for any $M\in\mathcal C$.

The conclusion follows by \cite[Proposition 2.2]{HT} and the assumption on the Bass modules over $\mathcal A^{\leq\omega}$.
\end{proof}

\medskip
We can now characterize the cotorsion pairs $(\mathcal A,\mathcal B)$ with both classes closed under direct limits among those which satisfy the condition only on the right-hand side. The characterization shows that when testing for $\varinjlim \mathcal A = \mathcal A$, it suffices to check only the Bass modules over $\mathcal A ^{\leq \omega}$: 

\begin{thm} \label{t:bigchar1} Let $R$ be a ring and $\mathfrak C=(\mathcal A,\mathcal B)$ a cotorsion pair in $\ModR$ \st $\mathcal B$ is closed under direct limits. Then the following conditions are equivalent:
\begin{enumerate}
\item $\mathfrak C$ is cogenerated by a (discrete) pure-injective module;
\item $\mathcal A$ is closed under pure-epimorphic images (and pure submodules);
\item $\mathfrak C$ is closed (i.e., $\varinjlim \mathcal A = \mathcal A$);
\item $\mathcal A$ contains all Bass modules over $\mathcal A ^{\leq \omega}$.
\end{enumerate}
\end{thm}
\begin{proof} $(1)\Rightarrow (2)$. Since $\mathcal A = {}^\perp C$ and $C$ is pure-injective, $\mathcal A$ is closed under pure-epimorphic images. 
In order to verify  closure under pure submodules, we take a pure submodule $X$ of a module $A$ from
$\mathcal A$ and  show that $\Hom_R(X,-)$ is exact on the short exact sequence $0\to C\to E(C)\to Z\to 0$ given by the injective envelope of $C$.
Since the first cosyzygy $Z$ of the pure-injective module $C$ is pure-injective (see e.g.\ \cite[Lemma 6.20]{GT}), every $f\in \Hom_R(X,Z)$ can be extended  to a homomorphism $f'\in \Hom_R(A,Z)$. Now $f'$ factors  through 
$E(C)$ as $A \in \mathcal A = {}^\perp C$. Restricting to $X$, we obtain the desired factorization.

The implications $(2)\Rightarrow (3)$ and $(3) \Rightarrow (4)$ are trivial.

$(4)\Rightarrow (1)$. By \cite[Theorem~6.1]{S}, $\mathfrak C$ is of countable type (whence $\mathfrak C$ is complete), and $\mathcal B$ is definable. We can apply \cite[Proposition 5.12]{SS} and obtain a set $\mathcal S$ of indecomposable pure-injective modules such that $\tilde{\mathcal A} = {}^\perp(\prod\mathcal S)$, where $\tilde{\mathcal A}$ denotes the class of all pure-epimorphic images of modules from $\mathcal A$. The direct product can be replaced by the pure-injective hull, $C$, of $\bigoplus \mathcal S$, which is a discrete pure-injective direct summand in $\prod\mathcal S$. Further, we can assume w.l.o.g. that $C$ cogenerates $\ModR$ (possibly replacing it by $C\oplus Q$ where $Q$ is an injective cogenerator).

Let $M \in \tilde{\mathcal A} = {}^\perp C$ be $\leq \kappa$-presented. By induction on $\kappa$, we will prove that $M \in \mathcal A$ (then $\mathcal A = {}^\perp C$, and $(1)$ will hold). The case of $\kappa = \aleph_0$ follows from $(4)$ by Lemma~\ref{l:countlim}. 

Assume that $\kappa$ is uncountable and all $<\!\kappa$-presented modules from $\tilde{\mathcal A}$ are in $\mathcal A$. Consider a free presentation of $M$, $$\begin{CD} 0 @>>> K @>>> R^{(\kappa)} @>f>> M @>>> 0. \end{CD}$$
Since $R^{(\kappa)}$ is (strict) $C$-stationary, so is $K$ by Lemma~\ref{l:kerepi} and Lemma~\ref{l:tilde}. Consider the family $\mathcal L$ for $K$ provided by Lemma~\ref{l:Cext}. By Corollary~\ref{c:Cstat}, where we take $\mathcal C=\{C\}$, $M$ is (strict) $C$-stationary as well; here, we use that all the modules from $\mathcal A$ are $C$-stationary. Using Lemma~\ref{l:Cext} again, we can build in $M$ an $\aleph_1$-dense system of countably presented submodules, w.l.o.g.\ of the form $f(R^{(I)})$ for a countable subset $I$ of $\kappa$. We denote this system by $\mathcal H$, and for every $H\in\mathcal H$, take a countable subset $I_H \subseteq \kappa$ \st $f(R^{(I_H)})=H$. 

Let $\mathcal M = \{H\in\mathcal H\mid \Ker(f\restriction R^{(I_H)})\in\mathcal L\}$. Then $\mathcal M$ is $\aleph _1$-dense and consists of modules from $\tilde{\mathcal A}$, by the property of the modules in $\mathcal L$ and by the assumption of $M\in {}^\perp C=\tilde{\mathcal A}$. So $\mathcal M\subseteq\mathcal A$ by the inductive premise for $\kappa = \aleph_0$. 

Next, we fix a continuous strictly ascending chain $(\delta_\gamma \mid 0<\gamma <\hbox{cf}(\kappa))$ of infinite ordinals $< \kappa$ which is cofinal in $\kappa$. Moreover, we put $\delta _0 = 0$ and $\delta _{\hbox{cf}(\kappa)} = \kappa$. Then we can easily build a continuous ascending chain $(\mathcal M_\gamma \mid \gamma\leq\hbox{cf}(\kappa))$ consisting of $\subseteq$-directed subsets of $\mathcal M$ such that $|\mathcal M_\gamma|=|\delta_\gamma|$ and $\bigcup \mathcal M_{\hbox{cf}(\kappa)} = M$. Understanding $\bigcup\varnothing$ as the trivial submodule of $M$, we claim that the continuous ascending chain $\mathcal F = (\bigcup \mathcal M_\gamma \mid \gamma\leq\hbox{cf}(\kappa))$ is actually an $\mathcal A$-filtration of $M$ (proving that $M\in\mathcal A$).

Indeed, all modules in $\mathcal F$ are elements of $\tilde{\mathcal A}$. Moreover, the property $(\dagger)$ from Lemma~\ref{l:Cext} is preserved when taking directed unions. It follows that all consecutive factors in $\mathcal F$ belong to $\tilde{\mathcal A} = {}^\perp C$, and hence to $\mathcal A$ by the inductive premise.
\end{proof}

If $\mathfrak C=(\mathcal A,\mathcal B)$ is a 
cotorsion pair of finite type, then the class $\mathcal B$ is definable (cf. \cite[Example 6.10]{GT}), so Theorem \ref{t:bigchar1} applies. In fact, the Bass modules over $\mathcal A ^{< \omega}$ are sufficient in this case  (see Lemma \ref{l:countlim}).

\begin{cor}\label{c:tilt} Let $R$ be a ring, 
and $\mathfrak C=(\mathcal A,\mathcal B)$ be a cotorsion pair of finite type.
Then $\mathfrak C$ is closed 
if and only if $\mathcal A$ contains all Bass modules over $\mathcal A ^{< \omega}$.
\end{cor} 

\bigskip
\section{Locally free modules and approximations}
\label{sec:lfree}

Now we can present several consequences for the structure of locally free modules. We start with the deconstructibility.  

Let $(\mathcal A,\mathcal B)$ be a cotorsion pair of countable type and let $\mathcal L$ denote the class of all locally $\mathcal A ^{\leq \omega}$-free modules. Assume there exists a Bass module $N$ over $\mathcal A ^{\leq \omega}$ such that $N \notin \mathcal A$. Then, by \cite[Theorem 6.2]{SlT}, the class $\mathcal L$ is not deconstructible. Thus, Theorem \ref{t:bigchar1} gives     

\begin{thm}\label{t:nondec} Let $R$ be a ring and $\mathfrak C = (\mathcal A,\mathcal B)$ a cotorsion pair with $\varinjlim \mathcal B = \mathcal B$. Let $\mathcal L$ be the class of all locally $\mathcal A ^{\leq \omega}$-free modules. Then $\mathcal L$ is deconstructible, if and only if $\mathcal A$ contains all Bass modules over $\mathcal A ^{\leq \omega}$.
\end{thm}   

\begin{cor}\label{c:nondectilt}  
If $T$ is a tilting module, then the class of all locally $T$-free modules is deconstructible, if and only if $T$ is $\sum$-pure-split, if and only if $\mathcal A$ contains all Bass modules over $\mathcal A ^{< \omega}$. 
\end{cor}
\begin{proof}
Every tilting cotorsion pair $\mathfrak C$ is of finite type, thus Corollary \ref{c:tilt} applies. Further, by \cite[Proposition 13.55]{GT}, $\mathfrak C$ is closed if and only if $T$ is $\Sigma$-pure-split.
\end{proof}
 
\begin{rem}\label{r:improves} Corollary~\ref{c:nondectilt} provides for a common generalization of several results from \cite[\S6]{SlT}, where the classes of locally $T$-free modules were shown not to be deconstructible for various instances of non-$\sum$-pure split tilting modules $T$. In particular, the assumption in \cite[\S6]{SlT} of $T$ being a direct sum of countably presented modules, turns out to be redundant.      
\end{rem} 

Let us now consider the existence of locally free precovers. The prototype case of flat Mittag-Leffler modules has already been treated in \cite[\S3]{S}; we now give a general answer for the tilting case. Note that we have no bound on the cardinality of the ring $R$.

\begin{thm} \label{t:tiltnprec} Let $R$ be a ring, $T$ a tilting module, and $\mathfrak C = (\mathcal A,\mathcal B)$ the tilting cotorsion pair induced by $T$. 
Then the class of all locally $T$-free modules is (pre)covering if and only if $T$ is $\Sigma$-pure-split. 
\end{thm}
\begin{proof} By Corollary~\ref{c:nondectilt}, if $T$ is not $\Sigma$-pure-split, then there exists a Bass module $N$ over $\mathcal A ^{< \omega}$ such that $N \notin \mathcal A$. Since $\mathcal A ^{\leq \omega}$ coincides with the class of all countably presented locally $T$-free modules, $N$ is not (a direct summand of) a locally $T$-free module. So $N$ satisfies the assumptions of \cite[Lemma~3.2]{S}, whence $N$ has no precover in the class of all locally $T$-free modules.

Conversely, if $T$ is $\Sigma$-pure split, then the class $\mathcal L$ of all locally $T$-free modules coincides with $\mathcal A$. Indeed, we have $\mathcal A\subseteq \mathcal L$ by \cite[Theorem 4.5]{SlT} and \cite[Theorem 7.13]{GT}, and the other inclusion follows by Corollaries~\ref{c:tilt} and \ref{c:nondectilt}, in particular the fact that $\mathcal A = \varinjlim\mathcal A$. Thus $\mathcal L$ is covering by e.g. \cite[Theorem 2.2.8]{X}.
\end{proof}

In fact, Theorem~\ref{t:tiltnprec} is an instance of even more general Theorem~\ref{t:bigchar2} and Corollary~\ref{c:P}. We can also express the local freeness in terms of $\mathcal B$-stationarity:  

\begin{thm} \label{t:localstat} Let $(\mathcal A,\mathcal B)$ be a cotorsion pair with $\mathcal B = \varinjlim \mathcal B$. Then a module $M$ is locally $\mathcal A^{\leq\omega}$-free, if and only if $M$ is a $\mathcal B$-stationary pure-epimorphic image of a module from $\mathcal A$.

In particular, the class of all locally $\mathcal A^{\leq\omega}$-free modules is closed under pure submodules.
\end{thm}
\begin{proof} The local $\mathcal A^{\leq\omega}$-freeness of $M$ just says that $M$ possesses an $\aleph_1$-dense system of submodules from $\mathcal A^{\leq \omega}$. Since $M$ is the directed union of these submodules, it is a pure-epimorphic image of their direct sum. Moreover, all modules in $\mathcal A$ are (strict) $\mathcal B$-stationary by \cite[Lemma~4.2]{S}, whence $M$ is $\mathcal B$-stationary by \cite[Corollary 2.6(5)]{H}.  

For the if-part, we take an elementary cogenerator $C$ for $\mathcal B$. Then $C$ is a cogenerator for $\ModR$, and the notions of $C$-stationarity and strict $C$-stationarity coincide (since $C$ is pure injective, cf.~\cite[Proposition 1.7]{H}). 
Next, we consider a free presentation of $M$,
$$\begin{CD} 0 @>>> K @>>> R^{(X)} @>f>> M @>>> 0. \end{CD}$$

Since $M$ is a pure-epimorphic image of a module in $\mathcal A$, it follows from Lemma~\ref{l:tilde} and Lemma~\ref{l:kerepi} that $M\in {}^\perp C$ and $K$ is strict $\mathcal B$-stationary. Applying Lemma~\ref{l:Cext}, we obtain an $\aleph _1$-dense system $\mathcal L$ consisting of strict $C$-stationary submodules of $K$. Since $M$ is $\mathcal B$-stationary, it is (strict) $C$-stationary. Using Lemma~\ref{l:Cext} again, this time for $M$ and $C$, we obtain an $\aleph_1$-dense system $\mathcal H$ of submodules in $M$, where each $H\in\mathcal H$ is w.l.o.g.\ of the form $f(R^{(X_H)})$ for a countable subset $X_H$ of $X$.

Now, the set $\mathcal M = \{N\in\mathcal H\mid \Ker(f\restriction R^{(X_H)})\in\mathcal L\}$ is an $\aleph _1$-dense system of submodules in $M$, and it consists of strict $C$-stationary countably presented modules from ${}^\perp C$ (use $(\dagger)$ for $\mathcal L$ and $M\in {}^\perp C$). By Proposition~\ref{p:pureinprod} and the definition of $C$, we infer $\mathcal M \subseteq \mathcal A^{\leq\omega}$, whence $M$ is locally $\mathcal A^{\leq\omega}$-free.

The final claim follows from the fact that a module $M$ is $\mathcal B$-stationary, if and only if it is (strict) $C$-stationary, see  \cite[Corollary 3.9]{AH}.  The latter property is inherited by pure submodules by \cite[Corollary 8.12(1)]{AH}. Moreover, the class of all pure-epimorphic images of modules from $\mathcal A$ equals ${}^\perp C$ by Lemma~\ref{l:tilde}, and as such, it is always closed under pure submodules (cf.~the proof of Theorem~\ref{t:bigchar1}).
\end{proof} 

\bigskip
\section{Covers and pure-injectivity}
\label{sec:covers}

In the late 1990s, Enochs asked whether each covering class of modules $\mathcal A$ is closed under direct limits (see e.g.\ \cite[Open Problems 5.4]{GT}). This problem is still open in general. We will give a positive answer  for the case when $\mathcal A$ fits in a cotorsion pair $\mathfrak C=(\mathcal A,\mathcal B)$ such that $\mathcal B$ is closed under direct limits. In particular, we have a positive answer when $\mathfrak C$ is any tilting cotorsion pair. 
 
We will approach the problem via an extension of Theorem \ref{t:bigchar1} by further equivalent conditions. First, we need a proposition of independent interest:

\begin{prop} \label{p:locsplit} Let $(\mathcal A,\mathcal B)$ be a cotorsion pair such that $B^{(\omega)}\in\mathcal B$ for every $B\in\mathcal B$. Assume that $h:A\to M$ is an $\mathcal A$-cover. Then $h$ is an isomorphism, if and only if the embedding $\,\Ker (h)\subseteq A$ is locally split, and $M^{(\omega)}$ has an $\mathcal A$-cover.
\end{prop}
\begin{proof} The only-if part is trivial. For the if-part, consider $x\in\Ker(h)$. By the hypothesis, there is a homomorphism $g:A\to\Ker(h)$ \st $g(x)=x$. Note that $h$ is a special $\mathcal A$-precover by the Wakamatsu lemma \cite[Lemma 5.13]{GT}. By our assumption on $\mathcal B$, the coproduct map $h^{(\omega)}:A^{(\omega)}\to M^{(\omega)}$ is a special $\mathcal A$-precover. We can now use \cite[Theorem 1.4.7]{X} (since the assumption of $M^{(\omega)}$ having an $\mathcal A$-cover, which is missing from its statement, is satisfied here) and obtain that $h^{(\omega)}$ is an $\mathcal A$-cover. By \cite[Theorem 1.4.1]{X}, we conclude that there exists $m < \omega$ \st $0=g^m(x)=x$. This proves that $\Ker(h) = 0$, whence $h$ is an isomorphism.
\end{proof}

\begin{thm} \label{t:bigchar2} Let $R$ be a ring and $\mathfrak C=(\mathcal A,\mathcal B)$ be a cotorsion pair in $\ModR$ \st $\mathcal B$ is closed under direct limits. Let $\mathcal L$ denote the class of all locally $\mathcal A ^{\leq\omega}$-free modules. Then the following conditions are equivalent:
\begin{enumerate}
\item $\mathfrak C$ is closed;
\item every module (in $\mathcal B$) has an $\mathcal A$-cover;
\item $\Ker (\mathfrak C )$ is closed under countable direct limits;
\item $(\varinjlim \mathcal A)^{\leq\omega}$ consists of (strict) $\mathcal B$-stationary modules;
\item $\Ker (\mathfrak C )$ consists of $\Sigma$-pure-split modules;
\item $\mathcal L$ coincides with the class $\tilde{\mathcal A}$ of all pure-epimorphic images of modules from~$\mathcal A$;
\item every module (in $(\varinjlim \mathcal A)^{\leq\omega}$) has an $\mathcal L$-(pre)cover.
\end{enumerate}
\end{thm} 
\begin{proof} First, we recall that by \cite[Theorem~6.1]{S}, the cotorsion pair $\mathfrak C$ is of countable type, hence it is complete, and $\mathcal B$ is a definable class. Further, by Lemma \ref{l:tilde}, there is an elementary cogenerator $C$ for $\mathcal B$ with $\tilde{\mathcal A}={}^\perp C$. Finally, note that, by \cite[Lemma~4.2]{S}, any module in $\mathcal A$ is strict $\mathcal B$-stationary.

$(1)\Rightarrow (2)$. This is well known: each precovering class closed under direct limits is covering, cf.\ \cite[Theorem 2.2.8]{X}.  

$(2)\Rightarrow (3)$. Let $M$ be a countable direct limit of modules from $\Ker (\mathfrak C) = \mathcal A\cap\mathcal B$.  Then there exist modules $M_i \in \Ker (\mathfrak C)$, morphisms $g_i : M_i \to M_{i+1}$, and a pure exact sequence 
$$0 \to \bigoplus _{i<\omega} M_i \overset{g}\to \bigoplus _{i<\omega} M_i \overset{h}\to M \to 0$$ 
such that $g \restriction M_i = \mbox{id}_{M_i} - g_i$ for each $i < \omega$,  e.g.~ by \cite[Lemma 2.12]{GT}. It is easy to see that $g$ is locally split.  

By our assumption on $\mathcal B$, $h$ is a special $\mathcal A$-precover of $M$, and $M\in\mathcal B$. By $(2)$ and \cite[Theorem 1.2.7]{X}, there is a direct summand $A$ of $\bigoplus _{i<\omega} M_i$ \st $h\restriction A$ is $\mathcal A$-cover of $M$, and $A\cap\Ker(h)$ is a direct summand in $\Ker(h)$. Note that the inclusion $A\cap\Ker(h)\subseteq A$ inherits the property of being locally split from $g$. By Proposition~\ref{p:locsplit}, we conclude that $M\cong A\in\mathcal A\cap\mathcal B$.

$(3)\Rightarrow (4)$. By Lemma~\ref{l:countlim}, each $M\in (\varinjlim \mathcal A)^{\leq\omega}$ is a direct limit of a countable direct system, $(A_n\mid n<\omega)$, of modules from $\mathcal A$. We expand this direct system, to a direct system of short exact sequences induced by special $\mathcal B$-preenvelopes. Its direct limit is a short exact sequence $0\to M\to B\to A\to 0$ where $B\in\mathcal A\cap\mathcal B$ by $(3)$, and $A\in\varinjlim\mathcal A$. Then $B$ is strict $\mathcal B$-stationary, and $A\in {}^\perp D$ for each pure-injective module $D \in \mathcal B$ by Lemma~\ref{l:tilde}. So $M$ is strict $\mathcal B$-stationary by Lemma~\ref{l:kerepi}.

$(4)\Rightarrow (1)$. Since $(\varinjlim \mathcal A)^{\leq\omega}\subseteq {}^\perp C$, it follows from $(4)$ and Proposition~\ref{p:pureinprod} that $(\varinjlim \mathcal A)^{\leq\omega}\subseteq\mathcal A$. So $\mathfrak C$ is closed by Theorem \ref{t:bigchar1}. 

$(1)\Rightarrow (5)$. This is clear, since $\mathcal B$ is definable, and $\mathcal A$ is closed under pure-epimorphic images by Theorem \ref{t:bigchar1}.

$(5)\Rightarrow (6)$. First, notice that $\mathcal B\cap\tilde{\mathcal A} \subseteq \mathcal A$: Indeed, if $M \in \mathcal B\cap\tilde{\mathcal A}$, and $f : A \to M$ is a special $\mathcal A$-precover of $M$, then $f$ is a pure epimorphism and $A \in \mathcal A\cap\mathcal B$, hence $f$ splits by $(5)$.   

Next, for a module $N\in\tilde{\mathcal A}$, we form a special $\mathcal B$-preenvelope $g : N \to B$. Then $B \in \mathcal A$ by the previous argument. In particular, $B$ is strict $\mathcal B$-stationary, and $N$ is strict $\mathcal B$-stationary by Lemma~\ref{l:kerepi}. So $N \in \mathcal L$ by Theorem \ref{t:localstat}. Conversely, $\mathcal L \subseteq \varinjlim \mathcal A ^{\leq\omega} \subseteq \tilde{\mathcal A}$, whence $(6)$ holds. 

$(6)\Rightarrow (7)$. This follows from the fact that $\tilde{\mathcal A}$ is a covering class.

$(7)\Rightarrow (4)$. Let $M \in (\varinjlim \mathcal A)^{\leq\omega}$. Then $M$ is a Bass module over $\mathcal A^{\leq\omega}$ by Lemma~\ref{l:countlim}. By $(7)$ and \cite[Lemma~3.2]{S}, $M$ is (a direct summand of) an element of $\mathcal L$, so $M$ is $\mathcal B$-stationary by Theorem \ref{t:localstat}.     
\end{proof}

\begin{rem} \label{r:C} Notice that by the proof above, the condition in $(4)$ can be relaxed further to assuming the stationarity with respect to a single (pure-injective) module~$C$. \end{rem}

As an application, we prove a generalization of \cite[Theorem 4.6]{AST}.

\begin{prop} \label{p:AST} Let $\mathfrak D = (\mathcal F,\mathcal G)$ be a cotorsion pair with $\mathcal F$ closed under direct limits and $\mathcal G$ closed under countable direct sums. Let us denote by $\mathcal C$ the class of all FP$_2$-modules from $\mathcal F$. Then $\mathfrak D$ is of finite type if and only if $\mathcal F = \varinjlim \mathcal C$.
\end{prop} 

\begin{proof} If $\mathfrak D$ is of finite type, then $\mathcal F = \varinjlim \mathcal C$ since $\mathcal F$ is closed under direct limits.

Conversely, let $\mathcal F = \varinjlim \mathcal C$. Let us denote by $\mathfrak C = (\mathcal A, \mathcal B)$ the cotorsion pair generated by the class $\mathcal C$. Then $\mathcal F$ coincides with the class $\tilde{\mathcal A}$ of all pure-epimorphic images of modules from $\mathcal A$. So, using Lemma~\ref{l:tilde}, $\mathcal F = {}^\perp C$ for an elementary cogenerator $C$ for $\mathcal B$.

To show that $\mathcal F = \mathcal A$, we verify the statement $(4)$ from Theorem~\ref{t:bigchar2}. By Remark~\ref{r:C} above, it is enough to observe that countably presented modules from $\mathcal F$ are $C$-stationary. However, it follows from Proposition~\ref{p:pureinprod} that they are $\mathcal G$-stationary, since $\mathcal G$ is closed under countable direct sums.
\end{proof}

In what follows, we will see that statements contained in Theorem~\ref{t:bigchar2} actually generalize (part of) the famous Theorem~P by Hyman Bass. The next lemma sheds more light on the structure of $\Ker (\mathfrak C)$, showing that it is quite similar to the tilting setting:

\begin{lem} \label{l:kernel} Let $\mathfrak C = (\mathcal A, \mathcal B)$ be a cotorsion pair with $\mathcal B$ closed under direct limits. Then $\Ker (\mathfrak C) = \Add (K)$ for a module $K$.
\end{lem}

\begin{proof} Put $\mu = |R|+\aleph _0$, and let $K$ be the direct sum of a representative set of all $\leq\!\mu$-presented modules in $\Ker (\mathfrak C)$. Let us denote this representative set by $\mathcal K$. We have to show that each $M\in\Ker (\mathfrak C)$ is (isomorphic to) a direct sum of $\leq\!\mu$-presented modules, which is equivalent to $M$ being $\mathcal K$-filtered.

We will prove this by induction on $\lambda$ where $\lambda$ is the minimal cardinal such that $M$ is $\lambda$-presented. It holds trivially for $\lambda\leq\mu$. Assume that $\lambda>\mu$ is regular. Since $M\in\mathcal A$, $M$ is $\mathcal A^{<\lambda}$-filtered by \cite[Theorem~6.1]{S} and \cite[Theorem 7.13]{GT}. Fix one such filtration of $M$ and use Lemma~\ref{l:hill} with $\mathcal S=\mathcal A^{<\lambda}$ to obtain a family $\mathcal F$ of submodules of $M$. Note that $\mathcal F\subseteq\mathcal A$. We build a continuous chain $(M_\alpha\mid \alpha\leq\lambda)$ of submodules of $M$ \st $M_0 = 0, M_\lambda = M$ and $M_\alpha$ are $<\!\lambda$-presented, for all $\alpha<\lambda$, as follows:

\smallskip

For $\alpha$ limit, we put $M_\alpha = \bigcup_{\beta<\alpha} M_\beta$. Let $\alpha=\beta + 1$. If $\alpha$ is odd, we define $M_\alpha$ as a member from $\mathcal F$ given by Lemma~\ref{l:hill}(4) for $N=0$ and $X=M_\beta$.

Now, let $\alpha$ be even. We choose $M_\alpha$ as a pure submodule of $M$ containing $M_\beta$ with $|M_\alpha| \leq |M_\beta|+\mu<\lambda$ (it is possible by \cite[Lemma 2.25(a)]{GT}). Then $M_\alpha\in\mathcal B$ since $M\in\mathcal B$ and $\mathcal B$ is definable. Note, however, that $M_\alpha$ need not be a member of~$\mathcal F$.

Consider the subchain $\mathcal C = (M_\alpha\mid \alpha\leq\lambda, \alpha\hbox{ is limit})$. Using the properties of $\mathcal B$ and $\mathcal F$, we see that each member of $\mathcal C$ as well as each of its consecutive factors belongs to $\Ker (\mathfrak C)$. The consecutive factors are $<\!\lambda$-presented, and we can use the inductive assumption to deduce that $M$ is $\mathcal K$-filtered.

\smallskip

If $\lambda$ is singular, we use \cite[Theorem 7.29]{GT}. As the sets $\mathcal S_\kappa$, $\mu<\kappa<\lambda$ regular, witnessing $\kappa$-$\mathcal K$-freeness of $M$, we choose $\mathcal F\cap\mathcal B^{<\kappa}$, where $\mathcal F$ is the family given by Lemma~\ref{l:hill} used for the regular cardinal $\kappa$ and $\mathcal S = \mathcal A^{<\kappa}$. It is straightforward to verify that $\mathcal S_\kappa$ has the desired properties stated in \cite[Definition 7.27]{GT}.
\end{proof}

Recall that a module $M$ has a \emph{perfect decomposition} if  it has a decomposition in modules with local endomorphism ring, and every module $N\in\Add (M)$ has a \emph{semiregular} endomorphism ring $S_N=\End_R(N)$, i.e., idempotents lift modulo the Jacobson radical $J(S_N)$, and $S_N/J(S_N)$ is a von Neumann regular ring.

The notion of perfect decomposition has many equivalent definitions, cf.\ \cite{AS}. In particular, $M$ has perfect decomposition 
if and only if every pure monomorphism from any direct sum $\bigoplus _{\alpha\in I}N_\alpha$ into a module from $\Add (M)$ splits whenever all of its finite subsums split. For example, every $\Sigma$-pure-split module has a perfect decomposition.

\medskip
Our next corollary extends (part of) Bass' Theorem P (which is the case of $\mathfrak C = (\mathcal P_0,\ModR)$ and $K=R$), and more in general, it also covers the tilting setting (for $K = T$ a tilting module). 

\begin{cor}\label{c:P} Let $R$ be a ring, and $\mathfrak C=(\mathcal A,\mathcal B)$ be a cotorsion pair with $\mathcal B = \varinjlim\mathcal B$. Let $K$ be a module \st $\Ker(\mathfrak C) = \Add(K)$ (see Lemma~\ref{l:kernel}). Then the following conditions are equivalent:
\begin{enumerate}
\item $\mathcal A = \varinjlim\mathcal A$;
\item every module (in $\mathcal B$) has an $\mathcal A$-cover;
\item every module in $\Ker(\mathfrak C)$ has a semiregular endomorphism ring;
\item $K$ has perfect decomposition;
\item $K$ is $\Sigma$-pure-split;
\item every module (in $\mathcal B$) has a $\Ker(\mathfrak C)$-cover.
\end{enumerate}
\end{cor} 
\begin{proof} The first two conditions are the same as in Theorem \ref{t:bigchar2}. Condition $(5)$ is equivalent to the same condition there. $(4)\Rightarrow (3)$ is trivial.

$(5)\Rightarrow (4)$ follows by the equivalent definition mentioned above.

$(3)\Rightarrow (6)$ follows from \cite[Proposition 4.1]{A} and the fact that $\Ker(\mathfrak C) = \Add(K)$.

$(6)\Rightarrow (2)$. Observe that any $\Ker(\mathfrak C)$-cover of a module $B\in\mathcal B$ is, in fact, an $\mathcal A$-cover: this follows, for instance, from \cite[Theorem 1.2.7]{X} applied to a special $\mathcal A$-precover (hence $\Ker(\mathfrak C)$-precover) of $B$.
\end{proof}

The condition $(5)$ above cannot be replaced by `$K$ is $\Sigma$-pure-injective.' While each $\Sigma$-pure-injective module is $\Sigma$-pure-split, the converse is not true in general, even for tilting modules: If $R$ is right perfect, then $K = R$ is certainly $\Sigma$-pure-split, but it need not be $\Sigma$-pure-injective. In fact, \cite[\S2]{Z1} contains an example of a right artinian ring which is not right pure-injective.       

There is, however, a case where $\Sigma$-pure splitting and $\Sigma$-pure-injectivity coincide, namely when the cotorsion pair $(\mathcal A,\mathcal B)$ has the property that $\mathcal A ^{< \omega}$ is covariantly finite. This is always true when $R$ is left noetherian and $\mathcal A ^{< \omega}$ consists of modules of projective dimension $\leq 1$ (see \cite[Proposition 2.5]{AHT2}). A further case is provided by the following observation.

\begin{lem}\label{l:covarfinhered}
Let $R$ be a left hereditary ring, and let $\mathcal C$ be a   class consisting of $FP_2$-modules, such that  $\mathcal C$ is closed under extensions, direct summands, and contains $R$. Then $\mathcal C$ is  covariantly finite.
\end{lem}
\begin{proof} By \cite[Theorem 8.40]{GT}, we can compute $\varinjlim\mathcal C$ as  $\Ker\Tor^R_1(-,\mathcal A)$ where $\mathcal A=\Ker\Tor^R_1(\mathcal C,-)$. Now $\mathcal A$ is closed under direct limits, and by the Ext-Tor-relations, it is the left-hand class of a cotorsion pair of left $R$-modules. By the assumption on $R$ and \cite[Lemma 9.7]{GT}, $\mathcal A = \varinjlim\mathcal A^{<\omega}$, whence $\varinjlim\mathcal C = \Ker\Tor^R_1(-,\mathcal A^{<\omega})$. It yields that $\varinjlim\mathcal C$ is closed under direct products. By the classic result of Crawley-Boevey, this amounts to $\mathcal C$ being covariantly finite in the class of all finitely presented right modules.  \end{proof}

We state the next auxiliary result in a more general setting:

\begin{lem} \label{l:covarfin} Let $\mathcal C$ be a covariantly finite subcategory of the category of all finitely presented modules, and $T\in\varinjlim\mathcal C$ be a module. Assume that every countable direct system in $\mathcal C$ is $T$-stationary. Then $T$ is $\Sigma$-pure-injective.
\end{lem}
\begin{proof} We will verify condition $(4)$ of \cite[Lemma~5.1]{S}. Let $M$ be an arbitrary countably presented module. We express $M$ as the direct limit of a countable direct system $(F_n,f_n\mid n<\omega)$ of finitely presented modules. Since $\mathcal C$ is covariantly finite, we can expand this direct system into a direct system of $\mathcal C$-preenvelopes $p_n:F_n\to C_n$.

Now, apply the contravariant Hom-functor $\Hom_R(-,T)$. From $T\in\varinjlim\mathcal C$, it follows (using Lenzing's result characterizing modules in $\varinjlim\mathcal C$) that $\Hom_R(p_n,T)$ are surjective maps. Since the inverse system $(\Hom_R(C_n,T))_{n<\omega}$ is Mittag-Leffler by our assumption, the inverse system $(\Hom_R(F_n,T))_{n<\omega}$ is the epimorphic image of a Mittag-Leffler inverse system. As such, it must be Mittag-Leffler as well by \cite[Proposition 13.2.1]{Gr}. In other words, $M$ is $T$-stationary.
\end{proof}

\begin{cor} \label{c:covarfin} Let $\mathfrak C = (\mathcal A,\mathcal B)$ be a cotorsion pair of finite type. Assume that $\mathcal A ^{<\omega}$ is covariantly finite in the category of all finitely presented modules.

Then the equivalent conditions of Theorems \ref{t:bigchar1} and \ref{t:bigchar2} are further equivalent to the condition that $\Ker (\mathfrak C)$ consists of $\Sigma$-pure-injective modules, 
which amounts to  $\Ker (\mathfrak C)=\Add(K)$ for a product-complete module $K$.
\end{cor}
\begin{proof}
If $\Ker (\mathfrak C)$ consists of $\Sigma$-pure-injective modules, then $(5)$ from Theorem \ref{t:bigchar2} trivially follows. For the opposite direction, assume condition $(4)$ of Theorem \ref{t:bigchar2} and use Lemma~\ref{l:covarfin} for $\mathcal C = \mathcal A ^{<\omega}$. (Notice that $\Ker (\mathfrak C) \subseteq \mathcal A \subseteq \varinjlim \mathcal C$ because $\mathfrak C$ is of finite type.) 

Finally, assume that $\mathfrak C$ is closed. Since $\mathfrak C$ is of finite type, we have $\mathcal A=\varinjlim\mathcal A^{<\omega}$, and we can use  Crawley-Boevey's result to deduce that $\mathcal A$  is closed under direct products. Thus $\Ker (\mathfrak C)$ is closed under direct products as well. Taking $K$ as in Lemma~\ref{l:kernel}, we see that $\Ker (\mathfrak C) = \Add (K)$ for a product-complete module $K$.
\end{proof}

\bigskip
\section{An application: pure-semisimple hereditary rings}
\label{sec:pss}


Our previous results allow us to give various characterizations of hereditary rings which are related to pure-semisimplicity. 

\begin{thm} \label{t:ringchar} Let $R$ be right hereditary. The  following statements are equivalent.
\begin{enumerate}
\item Every cotorsion pair in $\ModR$ is tilting.
\item Every cotorsion pair in $\ModR$ is of finite type.
\item All cotorsion pairs in $\ModR$ form a set.
\item Every class of right $R$-modules closed under transfinite extensions and direct summands and containing $R$ is deconstructible.
\item Every class of right $R$-modules closed under transfinite extensions and direct summands and containing $R$ is (special) precovering.
\item Every $1$-tilting right $R$-module is $\Sigma$-pure-split.
\end{enumerate}
Moreover, under these conditions, $R$ is right artinian.
\end{thm}
\begin{proof} $(1)\Rightarrow (2)$ holds by \cite[Theorem 13.46]{GT}. $(2)\Rightarrow (3)$ is trivial.

$(3)\Rightarrow (4)$ follows by contradiction: if such class $\mathcal C$ is not deconstructible, then $$\Bigl\{\Bigl({}^\perp\bigl({\mathcal C^{\leq\lambda}}^\perp\bigr),\mathcal {C^{\leq\lambda}}^\perp\Bigr) \mid \lambda\hbox{ infinite cardinal}\Bigr\}$$ is not a set. (Note that the closure properties of $\mathcal C$ imply that ${}^\perp\bigl({\mathcal C^{\leq\lambda}}^\perp\bigr)\subseteq\mathcal C$ for all $\lambda$, by \cite[Corollary 6.13]{GT}.)

$(4)\Rightarrow (5)$. The condition $(4)$ yields that any such class is the left-hand class of a cotorsion pair generated by a set (cf. \cite[Corollary 6.13]{GT}), hence the class is special precovering by \cite[Theorem 6.11(b)]{GT}.

$(5)\Rightarrow (6)$. Follows from Theorem~\ref{t:tiltnprec} and \cite[Theorem 4.5]{SlT}. 

$(6)\Rightarrow (1)$. Let $(\mathcal A, \mathcal B)$ be a cotorsion pair. Since $R$ is right hereditary,  the cotorsion pair $\mathfrak C$ generated by $\mathcal A^{<\omega}$ is $1$-tilting. By \cite[Lemma 9.7]{GT}, $\mathcal A\subseteq\varinjlim\mathcal A^{<\omega}$. The latter class, however, is the left-hand class of $\mathfrak C$ by $(6)$ and Corollary~\ref{c:P}, whence $\mathfrak C = (\mathcal A,\mathcal B)$.

\smallskip
Finally, assume that the equivalent conditions $(1)$--$(6)$ hold. By \cite[Theorem~3.3]{S}, it follows from $(5)$ that $R$ is right perfect. Further, the cotorsion pair generated by the class of all finitely presented modules is $1$-tilting, and so, by $(6)$, its left-hand class is closed under direct limits. It follows that the tilting class coincides with the class of all injective modules which is therefore closed under direct limits. This proves that $R$ is right noetherian.
\end{proof}

\begin{rem}\label{r:spi}
The implications $(1)\Rightarrow (2)\Rightarrow (3)\Rightarrow (4)\Rightarrow (5)\Rightarrow (6)$ hold true over any ring $R$. If $R$ is left hereditary, then $(6)$ implies that every $1$-tilting right $R$-module is $\Sigma$-pure-injective and even product-complete, see Corollary~\ref{c:covarfin}.
\end{rem}

Recall that a ring $R$ is \emph{left pure-semisimple} if all left $R$-modules are ($\Sigma$-)pure-injective. This is equivalent to the fact that every left $R$-module is a direct sum of finitely generated modules. Further, a module $M$ is a \emph{splitter} if $\Ext^1_R(M,M)=0$.

\begin{prop}\label{pss} \cite[Proposition 4.2]{key}, \cite[Remark 4.1(2)]{DG}  Let $R$ be a {left pure-semisimple} hereditary ring. Then every cotorsion pair in $\RMod$ is generated by a finitely generated tilting module, and every indecomposable (finitely generated) left $R$-module is a splitter.
\end{prop}

Observe that the second property  above then also holds for right modules.

\begin{rem}\label{r:right} Let $R$ be a {left pure-semisimple} hereditary ring. Then every  indecomposable finitely generated right $R$-module is a splitter. Indeed,  every such module $A$ is endofinite, and by \cite[Lemma 5.1(2)]{Z0}   it follows that $A \cong A^{++}$, where + denotes the local dual. Now we know that $A^+$, being an indecomposable finitely generated left module, is a direct summand of a tilting module $T$ with cotorsion pair $\mathfrak C=(\mathcal A, \mathcal B)$, so it is in the kernel of $\mathfrak C$. By \cite[Lemma 9.4 (3) and (5)]{AH}, it follows that $A\cong A^{++}$ is in the kernel of  the dual cotilting cotorsion pair $(\mathcal C, \mathcal D)$ in $\ModR$, so it is isomorphic to a direct summmand in a direct product of copies of a cotilting module, and thus it is a splitter. 
\end{rem}

As for the first property, we don't know whether it is left-right-symmetric. In fact, this would imply the long-standing Pure-Semisimplicity Conjecture.

\begin{prop}\label{pssconj}
Let $R$ be a {left pure-semisimple} hereditary ring. The following statements are equivalent.
\begin{enumerate}
\item $R$ has finite representation type.
\item $R$ satisfies (one of) the conditions in Theorem \ref{t:ringchar}. 
\item There are only finitely many tilting left modules up to equivalence.
\item There are only finitely many tilting right modules up to equivalence.
\end{enumerate}
\end{prop}
 \begin{proof} By \cite[Theorem 2.10]{wendofin} and \cite[Theorem 5.6]{key}, the ring $R$ has finite representation type if and only if a certain tilting left $R$-module $W$ is endofinite, or equivalently, a certain tilting right $R$-module $T$ satisfies the descending chain condition on cyclic $\End_R(T)$-submodules. But the latter follows from    condition (6) in Theorem \ref{t:ringchar}, since $T$ is then even $\Sigma$-pure-injective, cf.~Remark \ref{r:spi}. 
 
(3)$\Rightarrow$(1). Every tilting left module  is equivalent to a finitely generated one, so there are only finitely many indecomposable modules up to isomorphism that occur as direct summands in a tilting module. Since every (finitely generated) indecomposable left module is a splitter, thus a direct summand in some tilting module, the claim follows. 

(4)$\Rightarrow$(3). Since any tilting left module is product-complete, it is also cotilting, and thus equivalent to $T^c$ for some tilting right module $T$ by \cite[Theorems 15.31 and 15.18(a)]{GT}. The rest follows by \cite[Theorem 15.18(b)]{GT}.
\end{proof}


\bigskip
\section{An example: Bass modules for Lukas' tilting over hereditary artin algebras}
\label{sec:herartinalg}

We have seen that non-$\sum$-pure-split tilting modules $T$ give rise to non-decon\-structible, and even non-precovering, classes of all locally $T$-free modules. The point is the existence of a Bass module $N$ over $\mathcal A ^{< \omega}$ such that $N  \notin \mathcal A$ (see Corollary \ref{c:nondectilt} and Theorem \ref{t:tiltnprec}). 

We will now present a concrete construction of such a Bass module for the particular case when $R$ is a hereditary artin algebra of infinite representation type and $T$ is the Lukas tilting module (see \cite[Example 13.7(b)]{GT} and \cite[\S7]{SlT}). 

In this case, it is well known that the representative set of all indecomposable finitely generated modules can be divided in three parts: $\p$, $\q$, and $\tube$ formed by the the indecomposable preprojective, preinjective, and regular modules, respectively. The modules in the class $\add {(\p)}$ ($\add {(\q)}$, $\add {(\tube)}$) are called the finitely generated preprojective (preinjective, regular) modules.    

In our setting, $\mathcal A ^{< \omega} = \add {(\p)}$, while $\mathcal A$ is the class of all Baer (= $\p$-filtered) modules, $\mathcal L$ the class of all locally Baer modules, and $\varinjlim \mathcal A=\varinjlim\p $ coincides with the torsion-free class $\mathcal F$ in the  torsion pair $(\mathcal T, \mathcal F)$ generated by $\tube$, cf.\ \cite{AHT,AKT,SlT}.      

In this section, a strictly increasing chain of finitely generated modules $$\mathcal P : \quad 0 = P_0 \subsetneq P_1 \subsetneq \dots \subsetneq P_n \subsetneq P_{n+1} \subsetneq \dots$$ is called \emph{special} provided that $\mathcal P$ consists of preprojective modules, but $P_{n+1}/P_n \in $ is regular for each $0 < n < \omega$.     

\begin{lem}\label{chain} Let $\mathcal P$ be a special chain. Then the module $N = \bigcup_{n < \omega} P_n$ is a Bass module over $\mathcal A ^{< \omega}$ such that $N \notin \mathcal A$.
\end{lem}
\begin{proof} Clearly, $N \in \varinjlim_{\omega} \mathcal A ^{< \omega}$. 

Assume $N$ is a Baer module. As shown in \cite{AHT}, $N$ has a $\p$-filtration $( Q_n \mid n < \omega )$ consisting of finitely generated preprojective modules. Let $i < \omega$ be such that $P_1 \subseteq Q_i$. Then there is an epimorphism $N/P_1 \to N/Q_i$. However, this contradicts the fact that in the torsion pair $(\mathcal T, \mathcal F)$, the torsion-free class $\mathcal F$ contains all Baer modules, so $N/Q_i \in \mathcal F$, while $\mathcal T$ contains all finitely generated regular modules, so $N/P_1 \in \mathcal T$. This proves that $N \notin \mathcal A$.
\end{proof}
  
So it suffices to construct the special chains. We will distinguish the tame case from the wild one. For the latter, we will employ the notion of a mono orbit due to Dagmar Baer (see \cite{B} and \cite{KT}):

\begin{defn}\label{monoo} Let $R$ be a wild hereditary artin algebra, and $P \in \p$. 
The $\tau^{-1}$-orbit of $P$ (that is, the set $O(P) := \{ \tau^{-n}(P) \mid n < \omega \}$ where $\tau^{-1} = Tr D$ is the Auslander-Reiten translation) 
is called a \emph{mono orbit} provided that 
\begin{itemize}
\item[(a)] for each $X \in O(P)$ and each $Y \in \p$, all non-zero homomorphisms $f : X \to Y$ are injective; and
\item[(b)] every non-zero homomorphism between elements of $O(P)$ has a regular cokernel.            
\end{itemize}
\end{defn}

Baer proved that for each hereditary artin algebra of wild representation type, there exists an indecomposable projective module $P$ whose $\tau^{-1}$-orbit is a mono orbit (see \cite[Proposition 2.2]{B}). Of course, $O(Q)$ is then a mono orbit for each $Q \in O(P)$.       

\begin{lem}\label{tame-wild} 
\begin{itemize}

\item[(1)] Assume $R$ is tame. Let $P$ be a non-zero finitely generated preprojective module, and $\{ S_n \mid 0 < n < \omega \}$ be any sequence of simple regular modules from (not necessarily distinct) homogenous tubes. Then there exists a special chain $\mathcal P$ such that $P_1 = P$, and $P_{n+1}/P_n \cong S_n$ for each $0 < n < \omega$.
\item[(2)] Assume $R$ is wild. Let $P \in \p$ be such that $O(P)$ is a mono orbit. Then $O(P)$ contains a special chain such that $P_1 = P$.          
\end{itemize}
\end{lem}
\begin{proof} (1) Let $P_1 = P$, and assume that $P_n$ is defined for some $0 < n < \omega$. We can factorize the embedding of $P_n$ into its injective envelope through the homogenous tube containing $S_n$. Since the elements of the tube are $\{ S_n \}$-filtered, necessarily $\Hom _R(P_n,S_n) \neq 0$. As $\tau (S_n) = S_n$, the Auslander-Reiten formula gives $\Ext ^1_R(S_n,P_n) \neq 0$. So there is a non-split extension $0 \to P_n \to P_{n+1} \overset{g_n}\to S_n \to 0$. 

It remains to prove that $P_{n+1}$ is preprojective. Clearly, $P_{n+1}$ cannot have any non-zero preinjective direct summands. If $T_n$ is a non-zero regular direct summand of $P_{n+1}$, then $g_n(T_n) = S_n$, because $S_n$ is simple regular. Since $R$ is tame, the category of all finitely generated regular modules is abelian, so the kernel of $g_n \restriction T_n$ is the zero submodule of $P_n$. Thus $g_n$ splits, a contradiction.                             

(2) (Kerner) Consider $0 < k < \omega$ such that $\Hom _R(P,\tau^{-k}P) \neq 0$. Since $O(P)$ is a mono orbit, there is an exact sequence $0 \to P \to \tau^{-k}P \to X \to 0$ where $X$ is regular. Iterated application of the Auslander-Reiten translation $\tau^{-1}$ yields the exact sequences $0 \to \tau^{-nk}P \to \tau^{-(n+1)k}P \to \tau^{-n}X \to 0$ for all $0 < n < \omega$. Since the modules $\tau^{-n}X$ ($n < \omega$) are regular, it suffices to put $P_1 = P$ and $P_{n+1} = \tau^{-nk}P$ for each $0 < n < \omega$.     
\end{proof} 
                                       
\emph{Acknowledgement}: We thank Otto Kerner for suggesting the use of mono orbits for constructing special chains in the wild case.



\end{document}